\documentclass[11pt]{amsart}
\usepackage{amsmath,amsthm}
\usepackage{amstext,amsfonts,bm,amssymb}
\usepackage{graphics,graphicx} 
\usepackage{a4wide}
\usepackage{float}
\usepackage{color}
\usepackage{verbatim}

\usepackage{gastex}
 
\theoremstyle{plain}
\newtheorem{theorem}{Theorem}

\newtheorem{Lem}{Lemma}

\def\fl#1{\left\lfloor#1\right\rfloor}

\def\stif#1#2{\left[#1\atop#2\right]} 

\def\sttf2#1#2{\left[\!\!\left[#1\atop#2\right]\!\!\right]}  
\def\stf3f#1#2{\left[\!\!\left[\!\!\left[#1\atop#2\right]\!\!\right]\!\!\right]} 
\def\stff4#1#2{\left[\!\!\left[\!\!\left[\!\!\left[#1\atop#2\right]\!\!\right]\!\!\right]\!\!\right]}
\def\stss2#1#2{\left\{\!\!\left\{#1\atop#2\right\}\!\!\right\}}

\allowdisplaybreaks 

\begin{document}

\title[The symmetrized sum of finite $q$-multiple harmonic sums]{The symmetrized sum of finite $q$-multiple harmonic sums with one exceptional index}

\author{Hideaki Ishikawa}
\address{School of Education \\University of Toyama \\Toyama 930-8555\\ Japan}
\email{ishikawa@edu.u-toyama.ac.jp} 

\author{Takao Komatsu}
\address{
Department of Mathematics, Institute of Science Tokyo, 2-12-1 Ookayama, Meguro-ku, Tokyo 152-8551, Japan}
\email{komatsu@zstu.edu.cn;\, komatsu.t.al@m.titech.ac.jp} 
\thanks{T.K. is the corresponding author.}

\date{
}

\begin{abstract}
In this paper, we give explicit expressions for $q$-harmonic sums on $1-\cdots-1,A,1-\cdots-1$ indices. The case $A=1$ has been extensively studied, and a variety of identities, explicit evaluations, and structural properties are known. Likewise, numerous explicit evaluations are available for $q$-multiple zeta values and $q$-harmonic sums with uniform indices $A-\cdots-A$.  Although several methods are available for treating non-uniform indices, explicit evaluations for the pattern $1-\cdots-1,A,1-\cdots-1$ remain comparatively scarce. The case $A=2$ was settled only recently.  In this paper, we extend this framework to the general case $A\ge3$ as the total sum with the theory of the elementary symmetric functions, though we do not generally evaluate each summand at a prescribed position of $A$. We obtain explicit formulas in terms of binomial coefficients and degenerate Bernoulli numbers. Some expressions about $2-\cdots-2,A,2-\cdots-2$ are also shown. 
\medskip

\end{abstract}

\subjclass[2010]{Primary 11M32; Secondary 05A15, 05A19, 05A30, 11B37, 11B73}
\keywords{multiple zeta functions, elementary symmetric functions, roots of unity, degenerate Bernoulli polynomials}

\maketitle

\section{Introduction}

\label{sec:1}

For positive integers $s_1,s_2,\dots,s_m$, we consider the $q$-multiple harmonic sums of the form 
\begin{equation}
\mathfrak Z_n(q;;s_1,s_2,\dots,s_m):=\sum_{1\le i_1<i_2< \dots<i_m\le n-1}\frac{1}{(1-q^{i_1})^{s_1}(1-q^{i_2})^{s_2}\cdots(1-q^{i_m})^{s_m}}\,. 
\label{def:qssmzv}
\end{equation} 
In previous papers (\cite{CK,Ko25b,Ko25a}), when $s=s_1=s_2=\cdots=s_m$, we wrote $\mathfrak Z_n(q;m;s)=\mathfrak Z_n(q;;\underbrace{s,s,\dots,s}_m)$.    
When $n\to\infty$, the infinite version was studied by Schlesinger \cite{Schlesinger}:  
\begin{equation}
\mathfrak Z(q;;s_1,s_2,\dots,s_m):=\sum_{1\le i_1<i_2< \dots<i_m}\frac{1}{(1-q^{i_1})^{s_1}(1-q^{i_2})^{s_2}\cdots(1-q^{i_m})^{s_m}}
\label{def:qsmzv}
\end{equation}  
There are several different forms. See, e.g., \cite{Bradley,OOZ,Zhao,Zudilin}. For finite versions, see, e.g., \cite{BTT18,BTT20,Takeyama09,Tasaka21}.  

After multiplying (\ref{def:qssmzv}) by $(1-q)^{s_1+\cdots+s_m}$, and taking $q\to1$ in the finite sum (\ref{def:qssmzv}), then let $n\to\infty$ under the usual admissibility condition,  
the multiple zeta value of the form 
\begin{equation}
\zeta(s_1,s_2,\dots,s_m):=\sum_{1\le i_1<i_2<\dots<i_m}\frac{1}{i_1^{s_1}i_2^{s_2}\dots i_m^{s_m}}
\label{def:mzv}
\end{equation}
has been studied extensively, as well as its generalizations or modifications (see, e.g., \cite{Zhao_book}).  
 
Although $s_1,s_2,\dots,s_m$ may be arbitrary positive integers, considerable attention has been devoted to the case $s=s_1=s_2=\cdots=s_m$, that is, all the indices are the same. The same phenomenon occurs for (\ref{def:qssmzv}) and its various generalizations. In particular, several explicit formulas are now known when all indices coincide (\cite{Ko25b,Ko25a,KL,KP,KW}). 
Some of the simplest and most fundamental results (\cite{Ko25a}) are given as  
\begin{equation}
\mathfrak Z_n(\zeta_n;;\underbrace{1,\dots,1}_m)=\frac{1}{m+1}\binom{n-1}{m}
\label{eq:zz-1m}
\end{equation}
and by applying symmetric functions (\cite{Macdonald95}), as a determinant 
\begin{equation}
\mathfrak Z_n(\zeta_n;;s)=\left|\begin{array}{ccccc}
\frac{n-1}{2}&1&0&\cdots&\\ 
\frac{2}{3}\binom{n-1}{2}&\frac{n-1}{2}&1&&\vdots\\ 
\vdots&&\ddots&&0\\
\frac{s-1}{s}\binom{n-1}{s-1}&\frac{1}{s-1}\binom{n-1}{s-2}&\cdots&\frac{n-1}{2}&1\\ 
\frac{s}{s+1}\binom{n-1}{s}&\frac{1}{s}\binom{n-1}{s-1}&\cdots&\frac{1}{3}\binom{n-1}{2}&\frac{n-1}{2}\\ 
\end{array}
\right|\,.
\label{eq:zz-det2}
\end{equation}
Here, $\zeta_n=e^{2\pi\sqrt{-1}/n}$.   
In particular, by taking $s=2,3,\dots,9$ in (\ref{eq:zz-det2}), we have 
\begin{align*}
\mathfrak Z_n(\zeta_n;;2)&=-\frac{(n-1)(n-5)}{12}\,,\\
\mathfrak Z_n(\zeta_n;;3)&=-\frac{(n-1)(n-3)}{8}\,,\\
\mathfrak Z_n(\zeta_n;;4)&=\frac{(n-1)(n^3+n^2-109 n+251)}{6!}\,,\\
\mathfrak Z_n(\zeta_n;;5)&=\frac{(n-1)(n-5)(n^2+6 n-19)}{288}\,,\\
\mathfrak Z_n(\zeta_n;;6)&=-\frac{(n-1)(2 n^5+2 n^4-355 n^3-355 n^2+11153 n-19087)}{12\cdot 7!}\,,\\
\mathfrak Z_n(\zeta_n;;7)&=-\frac{(n-1)(n-7)(2 n^4+16 n^3-33 n^2-376 n+751)}{24\cdot 6!}\,,\\
\mathfrak Z_n(\zeta_n;;8)&=\frac{(n-1)(3 n^7+3 n^6-917 n^5-917 n^4+39697 n^3+39697 n^2-744383 n+1070017)}{10!}\,,\\
\mathfrak Z_n(\zeta_n;;9)&=\frac{27(n-1)(n-3)(n-9)(n^5+13 n^4+10 n^3-350 n^2-851 n+2857)}{2\cdot 10!}\,.
\end{align*}

In the case of uniform exponents, one can directly utilize the properties of Stirling numbers (\cite{Ko23,Ko24}), and Stirling number identities, and the expansions of generating functions derived from symmetric functions. However, in the case of non-uniform exponents, these approaches are not applicable; consequently, there are relatively few explicit results. 
Even the apparently simplest non-uniform case $1-\cdots-1,2,1-\cdots-1$ proved to be delicate: the relevant real-part identity was conjectured in \cite{BTT20} and proved in \cite{BIK}.

Recently, in \cite{BIK}, the case $1-\cdots-1,2,1-\cdots-1$ was settled by evaluating the reversal pair, equivalently the real part of each ordered sum. More precisely, for non-negative integers $a$ and $b$, the following was proved.  
For integers $n$ and $m$ with $n,m\ge 2$,  
$$ 
\mathfrak Z_n(\zeta_n;;\underbrace{\underbrace{1,\dots,1}_a,2,\underbrace{1,\dots,1}_b}_m)+\mathfrak Z_n(\zeta_n;;\underbrace{\underbrace{1,\dots,1}_b,2,\underbrace{1,\dots,1}_a}_m)=-\frac{m!(n-2 m-3)}{(m+2)!}\binom{n-1}{m}\,. 
$$ 
In other words, by taking the real part,  
$$
\mathfrak{Re}\left(\mathfrak Z_n(\zeta_n;;\underbrace{\underbrace{1,\dots,1}_a,2,\underbrace{1,\dots,1}_b}_m)\right)=-\frac{m!(n-2 m-3)}{2(m+2)!}\binom{n-1}{m}\,.
$$ 
In the present paper, we develop the approach of \cite{BIK} further, treat the general pattern $1-\cdots-1,A,1-\cdots-1$, and obtain explicit formulas for $A\ge3$, though we do not generally evaluate each summand at a prescribed position of $A$.

\section{Preliminaries}   

We first recall the following elementary facts from \cite{BIK}.  
 
Complex conjugation reverses the order of the indices. Consequently, the corresponding sums have the same real part and opposite imaginary parts for the index strings $s_1-s_2-\cdots-s_m$ and those $s_m-\cdots-s_2-s_1$.  

\begin{Lem} 
$$
\overline{\mathfrak Z_n(\zeta_n;;s_1,s_2,\dots,s_m)}=\mathfrak Z_n(\zeta_n;;s_m,s_{m-1},\dots,s_1)\,. 
$$ 
\label{lem:conj}  
\end{Lem}

For convenience, define 
\begin{equation}
u_r:=\frac{1}{1-\zeta_n^r}=\frac{1}{2}+\frac{\sqrt{-1}}{2}\cot\frac{r\pi}{n}\quad(1\le r\le n-1)\,.
\label{eq:ur}
\end{equation}
For a positive integer $A$, put the sum 
$$
P_m^{(A,j)}(n):=\mathfrak Z_n(\zeta_n;;\underbrace{1,\dots,1}_{j-1},A,\underbrace{1,\dots,1}_{m-j})=\sum_{1\le i_1<\dots<i_m\le n-1}u_{i_1}\cdots u_{i_{j-1}}u_{i_j}^A u_{i_{j+1}}\cdots u_{i_m}\,. 
$$ 
We also consider the average of $P_m^{(A,j)}(n)$ over the position of the exceptional index: 
$$
Q_m^{(A)}(n)=\frac{1}{m}\sum_{j=1}^m P_m^{(A,j)}(n)\,. 
$$  
By using the elementary symmetric function, we write 
$$
e_k:=e_k(u_1,\dots,u_{n-1})=\sum_{1\le i_1<\dots<i_k\le n-1}u_{i_1}\cdots u_{i_k}\,. 
$$ 
For a positive integer $A$, we also write 
$$
e_k^{(A)}:=e_k(u_1^A,\dots,u_{n-1}^A)\,,  
$$ 
so that $e_k=e_k^{(1)}$.
Then we have 
\begin{align}
e_1^{(A-1)}e_m&=\sum_{j=1}^{n-1}u_j^{A-1}\sum_{1\le i_1<\dots<i_m\le n-1}u_{i_1}\cdots u_{i_m}\notag\\
&=\sum_{j=1}^{m+1}P_{m+1}^{(A-1,j)}(n)+\sum_{|S|=m}\left(\sum_{j\in S}u_j^{A-1}\right)\prod_{h\in S}u_h\notag\\
&=\sum_{j=1}^{m+1}P_{m+1}^{(A-1,j)}(n)+\sum_{j=1}^{m}P_{m}^{(A,j)}(n)\,. 
\label{eq:eap}
\end{align} 
For a given $(m+1)$-element subset $H$, every monomial $\prod_{h\in H}u_h^{a_h}$ with one exponent equal to $A$ and the remaining $m$ exponents equal to $1$ occurs in $e_1^{(A-1)}e_m$ exactly $m+1$ times, once for each choice of the element of $H$ carrying the exponent $A$.

\section{Main results}   

We now obtain the following explicit polynomial evaluations.   

\begin{theorem}  
For $n,m\ge 1$, we have 
\begin{align*}
&\sum_{j=1}^{m}\mathfrak Z_n(\zeta_n;;\underbrace{1,\dots,1}_{j-1},2,\underbrace{1,\dots,1}_{m-j})\\
&=-\frac{m(n-2 m-3)}{2(m+1)(m+2)}\binom{n-1}{m}\,,\\
&\sum_{j=1}^{m}\mathfrak Z_n(\zeta_n;;\underbrace{1,\dots,1}_{j-1},3,\underbrace{1,\dots,1}_{m-j})\\
&=\frac{1}{m+1}\binom{n-1}{m}\left(-\frac{(n-1)(n-5)}{12}+\frac{(m+1)(n-2 m-5)(n-m-1)}{2(m+2)(m+3)}\right)\,,\\
&\sum_{j=1}^{m}\mathfrak Z_n(\zeta_n;;\underbrace{1,\dots,1}_{j-1},4,\underbrace{1,\dots,1}_{m-j})\\
&=-\frac{(n-1)(n-3)}{8(m+1)}\binom{n-1}{m}+\frac{(n-1)(n-5)}{12(m+2)}\binom{n-1}{m+1}-\frac{(m+2)(n-2 m-7)}{2(m+3)(m+4)}\binom{n-1}{m+2}\,,\\
&\sum_{j=1}^{m}\mathfrak Z_n(\zeta_n;;\underbrace{1,\dots,1}_{j-1},5,\underbrace{1,\dots,1}_{m-j})\\
&=\frac{(n-1)(n^3+n^2-109 n+251)}{6!(m+1)}\binom{n-1}{m}+\frac{(n-1)(n-3)}{8(m+2)}\binom{n-1}{m+1}\\
&\quad -\frac{(n-1)(n-5)}{12(m+3)}\binom{n-1}{m+2}+\frac{(m+3)(n-2 m-9)}{2(m+4)(m+5)}\binom{n-1}{m+3}\,.
\end{align*}
\label{th:1}
\end{theorem}

\noindent 
{\it Remark.}   
For $A=2$, the real part of $\mathfrak Z_n(\zeta_n;;\underbrace{1,\dots,1}_{j-1},A,\underbrace{1,\dots,1}_{m-j})$ is the same for any integer $j$ with $1\le j\le m$. However, for $A\ge 3$, reversal still gives equal real parts, but the real part need not be independent of the position $j$. 
Hence, each expression of the real part of $\mathfrak Z_n(\zeta_n;;\underbrace{1,\dots,1}_{j-1},A,\underbrace{1,\dots,1}_{m-j})$ depends on the value $j$.

\begin{proof}
From \cite{BIK}, we know that for any $j$ with $1\le j\le m$, the real part of $P_m^{(2,j)}(n)$ is given as 
$$
-\frac{m!(n-2 m-3)}{2(m+2)!}\binom{n-1}{m}
$$
and the imaginary parts of $P_m^{(2,j)}(n)$ and $-P_m^{(2,m-j+1)}(n)$ are the same. Hence,  
$$
\sum_{j=1}^{m+1}P_{m+1}^{(2,j)}(n)=-\frac{(m+1)(m+1)!(n-2 m-5)}{2(m+3)!}\binom{n-1}{m+1}\,. 
$$ 
By (\ref{eq:zz-1m}) and (\ref{eq:zz-det2}) with $s=2$, 
we see that 
$$
e_m=\frac{1}{m+1}\binom{n-1}{m}
$$
and 
$$
e_1^{(2)}=\mathfrak Z_n(\zeta_n;;2)=-\frac{(n-1)(n-5)}{12}\,. 
$$ 
Hence, by (\ref{eq:eap}), we have  
\begin{align}  
&\sum_{j=1}^{m}P_{m}^{(3,j)}(n)=e_1^{(2)}e_m-\sum_{j=1}^{m+1}P_{m+1}^{(2,j)}(n)\notag\\
&=-\frac{(n-1)(n-5)}{12}\cdot\frac{1}{m+1}\binom{n-1}{m}+\frac{(m+1)(m+1)!(n-2 m-5)}{2(m+3)!}\binom{n-1}{m+1}\notag\\
&=-\frac{(n-1)(n-5)}{12(m+1)}\binom{n-1}{m}+\frac{(m+1)(n-2 m-5)}{2(m+2)(m+3)}\binom{n-1}{m+1}\,.
\label{eq:pm3}  
\end{align}

We next consider $A=4$.   
By (\ref{eq:zz-det2}) with $s=3$, 
we see that 
$$
e_1^{(3)}=\mathfrak Z_n(\zeta_n;;3)=-\frac{(n-1)(n-3)}{8}\,. 
$$ 
Hence, by (\ref{eq:eap}) and (\ref{eq:pm3}), we have  
\begin{align}  
&\sum_{j=1}^{m}P_{m}^{(4,j)}(n)=e_1^{(3)}e_m-\sum_{j=1}^{m+1}P_{m+1}^{(3,j)}(n)\notag\\
&=-\frac{(n-1)(n-3)}{8}\cdot\frac{1}{m+1}\binom{n-1}{m}\notag\\
&\quad+\frac{(n-1)(n-5)}{12(m+2)}\binom{n-1}{m+1}-\frac{(m+2)(n-2 m-7)}{2(m+3)(m+4)}\binom{n-1}{m+2}\notag\\
&=-\frac{(n-1)(n-3)}{8(m+1)}\binom{n-1}{m}+\frac{(n-1)(n-5)}{12(m+2)}\binom{n-1}{m+1}-\frac{(m+2)(n-2 m-7)}{2(m+3)(m+4)}\binom{n-1}{m+2}\,.
\label{eq:pm4}   
\end{align}

Finally, let $A=5$.   
By (\ref{eq:zz-det2}) with $s=4$, 
we see that 
$$
e_1^{(4)}=\mathfrak Z_n(\zeta_n;;4)=\frac{(n-1)(n^3+n^2-109 n+251)}{6!}\,. 
$$ 
Hence, by (\ref{eq:eap}) and (\ref{eq:pm4}), we have  
\begin{align}  
&\sum_{j=1}^{m}P_{m}^{(5,j)}(n)=e_1^{(4)}e_m-\sum_{j=1}^{m+1}P_{m+1}^{(4,j)}(n)\notag\\
&=\frac{(n-1)(n^3+n^2-109 n+251)}{6!}\cdot\frac{1}{m+1}\binom{n-1}{m}\notag\\
&\quad+\frac{(n-1)(n-3)}{8(m+2)}\binom{n-1}{m+1}-\frac{(n-1)(n-5)}{12(m+3)}\binom{n-1}{m+2}+\frac{(m+3)(n-2 m-9)}{2(m+4)(m+5)}\binom{n-1}{m+3}\notag\\
&=\frac{(n-1)(n^3+n^2-109 n+251)}{6!(m+1)}\binom{n-1}{m}+\frac{(n-1)(n-3)}{8(m+2)}\binom{n-1}{m+1}\notag\\
&\quad -\frac{(n-1)(n-5)}{12(m+3)}\binom{n-1}{m+2}+\frac{(m+3)(n-2 m-9)}{2(m+4)(m+5)}\binom{n-1}{m+3}\,.
\label{eq:pm5}   
\end{align}
\end{proof}

Iterating (\ref{eq:eap}) and using the depth-one evaluations above gives explicit polynomial expressions for these symmetrized sums for $A=6,7,\ldots$. The following subsection provides a uniform expression in terms of degenerate Bernoulli numbers.  A uniform polynomial formula remains less transparent, since a general closed form of $\mathfrak Z_n(\zeta_n;;s)$ yet.

\subsection{Degenerate Bernoulli numbers}   

Using degenerate Bernoulli numbers, we obtain a uniform expression for these symmetrized sums $\sum_{j=1}^{m}P_{m}^{(A,j)}(n)$.  

The degenerate Bernoulli polynomials $\beta_n(x|\lambda)$ are defined by Carlitz \cite{Carlitz79} as 
\begin{equation}
\frac{t(1+\lambda t)^{x/\lambda}}{(1+\lambda t)^{1/\lambda}-1}=\sum_{k=0}^\infty\beta_k(x|\lambda)\frac{t^k}{k!}\,.  
\label{def:dbp}
\end{equation}
When $x=0$ in (\ref{def:dbp}), 
the degenerate Bernoulli numbers $\beta_n(\lambda)=\beta_n(0|\lambda)$ are defined by Carlitz \cite{Carlitz56} as 
\begin{equation}
\frac{t}{(1+\lambda t)^{1/\lambda}-1}=\sum_{k=0}^\infty\beta_k(\lambda)\frac{t^k}{k!}\,.  
\label{def:dbn}
\end{equation}
When $\lambda\to 0$, this generating function becomes $t/(e^t-1)$, which is that of the classical Bernoulli numbers $B_n=\lim_{\lambda\to 0}\beta_n(\lambda)$.   One of the expressions of degenerate Bernoulli numbers is given by Howard \cite[Theorem 3.1]{Howard96} as 
$$
\beta_m(\lambda)=C_m\lambda^m+\sum_{j=1}^{\fl{\frac{m}{2}}}\frac{m}{2 j}B_{2 j}\stif{m-1}{2 j-1}(-\lambda)^{m-2 j}\quad(m\ge 2)
$$
with 
$$
\beta_0(\lambda)=1\quad\hbox{and}\quad \beta_1(\lambda)=-\frac{1}{2}+\frac{\lambda}{2}\,.
$$ 
Here, $\stif{n}{k}$ denotes the (unsigned) Stirling numbers of the first kind, defined by   
$$
x(x-1)(x-2)\cdots(x-n+1)=\sum_{k=0}^n(-1)^{n-k}\stif{n}{k}x^k\,. 
$$ 
$C_m$ are Cauchy numbers ($b_m=C_m/m!$ are called Bernoulli numbers of the second kind), defined by the generating function 
$$
\frac{t}{\log(1+t)}=\sum_{n=0}^\infty C_n\frac{t^n}{n!} 
$$ 
\cite{Comtet,Ko13,MSV}. One of the expressions of Cauchy numbers is 
$$
C_n=\sum_{k=0}^n\stif{n}{k}\frac{(-1)^{n-k}}{k+1}\quad(n\ge 1)\,. 
$$ 
In \cite[Theorem 4]{Ko25a}, the determinant expression of $\mathfrak Z_n(\zeta_n;;s)$ in (\ref{eq:zz-det2}) is also given in terms of degenerate Bernoulli numbers as 
\begin{equation}  
\mathfrak Z_n(\zeta_n;;s)=-\sum_{j=1}^s\binom{s-1}{j-1}\beta_j(n^{-1})\frac{n^j}{j!}\,. 
\label{eq:mzv1s-degber}
\end{equation}

For simplicity, for any positive integer $A$, put 
$$
\mathcal F_m^{(A)}=\sum_{j=1}^m P_m^{(A,j)}(n)\,. 
$$ 
We know that 
$$
\mathcal F_m^{(2)}=-\frac{m(n-2 m-3)}{2(m+1)(m+2)}\binom{n-1}{m}\quad\hbox{and}\quad e_m=\frac{1}{m+1}\binom{n-1}{m}\,. 
$$   
From (\ref{eq:mzv1s-degber}), we see that 
$$
e_1^{(A)}=-\sum_{j=1}^A\binom{A-1}{j-1}\beta_j(n^{-1})\frac{n^j}{j!}\,. 
$$ 
As discussed above, we get 
\begin{align*}  
\mathcal F_m^{(3)}&=e_1^{(2)}e_m-\mathcal F_{m+1}^{(2)}\,,\\ 
\mathcal F_m^{(4)}&=e_1^{(3)}e_m-\mathcal F_{m+1}^{(3)}\\
&=e_1^{(3)}e_m-e_1^{(2)}e_{m+1}+\mathcal F_{m+2}^{(2)}\,,\\
\mathcal F_m^{(5)}&=e_1^{(4)}e_m-e_1^{(3)}e_{m+1}+e_1^{(2)}e_{m+2}-\mathcal F_{m+3}^{(2)}\,. 
\end{align*}   
Applying (\ref{eq:eap}) repeatedly, we have 
\begin{align*} 
\mathcal F_m^{(A)}&=\sum_{k=0}^{A-3}(-1)^k e_1^{(A-k-1)}e_{m+k}+(-1)^A\mathcal F_{m+A-2}^{(2)}\\
&=\sum_{k=0}^{A-3}(-1)^k(-1)\sum_{j=1}^{A-k-1}\binom{A-k-2}{j-1}\beta_j(n^{-1})\frac{n^j}{j!}\frac{1}{m+k+1}\binom{n-1}{m+k}\\
&\quad+(-1)^A(-1)\frac{(m+A-2)(n-2 m-2 A+1)}{2(m+A-1)(m+A)}\binom{n-1}{m+A-2}\\
&=\sum_{k=0}^{A-3}\sum_{j=1}^{A-k-1}\frac{(-1)^{k+1}n^j\beta_j(n^{-1})}{(m+k+1)j!}\binom{n-1}{m+k}\binom{A-k-2}{j-1}\\
&\quad+(-1)^{A+1}\frac{(m+A-2)(n-2 m-2 A+1)}{2(m+A-1)(m+A)}\binom{n-1}{m+A-2}\,.
\end{align*}

We thus obtain the following explicit expression in terms of the degenerate Bernoulli numbers. The binomial coefficients outside their natural range are recognized as zero.    

\begin{theorem}  
For an integer $A$ with $A\ge 2$ and positive integers $n$ and $m$, we have 
\begin{align*}
&\sum_{j=1}^{m}\mathfrak Z_n(\zeta_n;;\underbrace{1,\dots,1}_{j-1},A,\underbrace{1,\dots,1}_{m-j})\\
&=\sum_{k=0}^{A-3}\sum_{j=1}^{A-k-1}\frac{(-1)^{k+1}n^j\beta_j(n^{-1})}{(m+k+1)j!}\binom{n-1}{m+k}\binom{A-k-2}{j-1}\\
&\quad+(-1)^{A+1}\frac{(m+A-2)(n-2 m-2 A+1)}{2(m+A-1)(m+A)}\binom{n-1}{m+A-2}\,.
\end{align*}
\label{th:2}
\end{theorem}

\subsection{Another expression of $\mathfrak Z_n(\zeta_n;;s)$ in terms of degenerate Bernoulli polynomials}  

The degenerate Bernoulli polynomials provide a natural framework for a more compact formulation. The following result may also be useful in future extensions to more general index patterns.  
Theorem \ref{th:mzv1s-degberpol} expresses the depth-one sum through degenerate Bernoulli polynomials in (\ref{def:dbp}) without the finite sum in (\ref{eq:mzv1s-degber}).  

\begin{theorem}  
For positive integers $n$ and $s$, we have 
$$
\mathfrak Z_n(\zeta_n;;s)=\frac{(-1)^s n^{s+1}}{(s+1)!}\left(\beta_{s+1}\left(\frac{n-1}{n}\left|\frac{1}{n}\right.\right)-\beta_{s+1}\left(0\left|\frac{1}{n}\right.\right)\right)\,.  
$$ 
\label{th:mzv1s-degberpol}
\end{theorem}  

To prove Theorem \ref{th:mzv1s-degberpol}, we need the following identity. 
Here, define $\mathfrak Z_n(\zeta_n;;0)=n-1$. The identity is one of formal power series for $|x|<2\sin(\pi/n)$. 

\begin{Lem}  
For an integer $n$ with $n\ge 2$, we have 
$$
\sum_{s=0}^\infty\mathfrak Z_n(\zeta_n;;s)x^s=\frac{n\bigl(1-(1-x)^{n-1}\bigr)}{1-(1-x)^n}\,.
$$ 
\label{lem:mzv1s-degberpol}
\end{Lem} 
\begin{proof}  
We have 
\begin{align*}
\sum_{s=0}^\infty\mathfrak Z_n(\zeta_n;;s)x^s&=\sum_{s=0}^\infty\sum_{i=1}^{n-1}\frac{x^s}{(1-\zeta_n^i)^s}\\
&=\sum_{i=1}^{n-1}\frac{1-\zeta_n^i}{1-\zeta_n^i-x}=\sum_{i=0}^{n-1}\frac{1-\zeta_n^i}{1-\zeta_n^i-x}\,. 
\end{align*}
Now, by putting $y=1-x$, we shall prove that 
\begin{equation}  
\sum_{i=0}^{n-1}\frac{1-\zeta_n^i}{1-\zeta_n^i-x}=\frac{n\bigl(1-y^{n-1}\bigr)}{1-y^n}\,.
\label{eq:mzv1s-degberpol}  
\end{equation}
Since 
$$
\log(y^n-1)=\log\prod_{i=0}^{n-1}(y-\zeta_n^i)=\sum_{i=0}^{n-1}\log(y-\zeta_n^i)\,, 
$$ 
by differentiating both sides with respect to $y$, we get 
\begin{equation}  
\frac{n y^{n-1}}{y^n-1}=\sum_{i=0}^{n-1}\frac{1}{y-\zeta_n^i}\,.
\label{eq:mzv1s-01}  
\end{equation}
Next, define $t_i=\zeta_n^i$ ($i=0,1,\dots,n-1$) and write $P'(t_i)=n t_i^{n-1}=n/t_i$. Then, the equality $P'(t)=n/t$ is valid at these roots. 
Hence, using the partial fractions identity, 
$$
\frac{1}{t^n-1}=\sum_{i=0}^{n-1}\frac{1}{P'(t_i)(t-t_i)}=\sum_{i=0}^{n-1}\frac{1}{n/t_i(t-t_i)}=\frac{1}{n}\sum_{i=0}^{n-1}\frac{t_i}{t-t_i}\,. 
$$ 
Hence, 
\begin{equation} 
\frac{n}{y^n-1}=\sum_{i=0}^{n-1}\frac{\zeta_n^i}{y-\zeta_n^i}\,. 
\label{eq:mzv1s-02}  
\end{equation} 
By subtracting (\ref{eq:mzv1s-02}) from (\ref{eq:mzv1s-01}) side by side, we obtain (\ref{eq:mzv1s-degberpol}). 
\end{proof} 

\begin{proof}[Proof of Theorem \ref{th:mzv1s-degberpol}.]  
By setting 
$$
(x,\lambda,t)=\left(\frac{n-1}{n},\frac{1}{n},-n x\right)\quad\hbox{and}\quad \left(0,\frac{1}{n},-n x\right)
$$ 
in degenerate Bernoulli polynomials in (\ref{def:dbp}), we have 
$$
\frac{-n x(1-x)^{n-1}}{(1-x)^n-1}=\sum_{s=0}^\infty\beta_{s}\left(\frac{n-1}{n}\left|\frac{1}{n}\right.\right)\frac{(-n)^s}{s!}x^s
$$ 
and 
$$
\frac{-n x}{(1-x)^n-1}=\sum_{s=0}^\infty\beta_{s}\left(0\left|\frac{1}{n}\right.\right)\frac{(-n)^s}{s!}x^s\,,
$$ 
respectively. Subtracting both sides of these two identities gives 
$$
\frac{-n x\bigl((1-x)^{n-1}-1\bigr)}{(1-x)^n-1}=\sum_{s=0}^\infty\frac{(-1)^s n^s}{s!}\left(\beta_{s}\left(\frac{n-1}{n}\left|\frac{1}{n}\right.\right)-\beta_{s}\left(0\left|\frac{1}{n}\right.\right)\right)x^s\,. 
$$ 
Hence, 
$$
\frac{n\bigl((1-x)^{n-1}-1\bigr)}{(1-x)^n-1}=\sum_{s=0}^\infty\frac{(-1)^s n^{s+1}}{(s+1)!}\left(\beta_{s+1}\left(\frac{n-1}{n}\left|\frac{1}{n}\right.\right)-\beta_{s+1}\left(0\left|\frac{1}{n}\right.\right)\right)x^s\,. 
$$ 
Together with Lemma \ref{lem:mzv1s-degberpol}, by comparing the coefficients on both sides, we get the desired result.  
\end{proof}

\section{Some expressions about $2-\cdots-2,A,2-\cdots-2$}   

For convenience, define  
$$
R_m^{(A,j)}(n)=\sum_{1\le i_1<\dots<i_m\le n-1}u_{i_1}^2\cdots u_{i_{j-1}}^2 u_{i_j}^A u_{i_{j+1}}^2\cdots u_{i_m}^2\,.
$$
Then consider the sum 
$$
\mathcal R_m^{(A)}=\sum_{j=1}^m R_m^{(A,j)}(n)\,. 
$$
In this setting, we have 
\begin{align}  
e_1^{(A-1)}e_m^{(2)}&=(u_1^{A-1}+\cdots+u_{n-1}^{A-1})\left(\sum_{1\le i_1<\dots<i_m\le n-1}u_{i_1}^2\cdots u_{i_m}^2\right)\notag\\
&=\mathcal R_{m+1}^{(A-1)}+\mathcal R_m^{(A+1)}\,. 
\label{eq:22a22}
\end{align}  

From \cite[Theorem 5]{Ko25a}, we find that 
$$
e_m^{(2)}=\frac{1}{n(m+1)}\left(\binom{n-1}{m}+(-1)^m\binom{n-1}{2 m+1}\right)\,.
$$ 
Also, for any $j$ with $1\le j\le m$, $R_m^{(2,j)}(n)=e_m^{(2)}$.  Hence, when $A=3$, together with the fact that $e_1^{(2)}=-(n-1)(n-5)/12$, by applying (\ref{eq:22a22}) we have  
\begin{align*}  
\mathcal R_m^{(4)}&=e_1^{(2)}e_m^{(2)}-(m+1)e_{m+1}^{(2)}\\
&=-\frac{(n-1)(n-5)}{12}\frac{1}{n(m+1)}\left(\binom{n-1}{m}+(-1)^m\binom{n-1}{2 m+1}\right)\\
&\quad -\frac{m+1}{n(m+2)}\left(\binom{n-1}{m+1}+(-1)^{m+1}\binom{n-1}{2 m+3}\right)\\
&=-\frac{(n-1)(n-5)}{12 n(m+1)}\left(\binom{n-1}{m}+(-1)^m\binom{n-1}{2 m+1}\right)\\
&\quad -\frac{m+1}{n(m+2)}\left(\binom{n-1}{m+1}+(-1)^{m+1}\binom{n-1}{2 m+3}\right)\,.
\end{align*} 
Applying (\ref{eq:22a22}) as $A=5$ and $A=7$, we have 
\begin{align*}  
\mathcal R_m^{(6)}&=e_1^{(4)}e_m^{(2)}-\mathcal R_{m+1}^{(4)}\\
&=\frac{(n-1)(n^3+n^2-109 n+251)}{6!n(m+1)}\left(\binom{n-1}{m}+(-1)^{m}\binom{n-1}{2 m+1}\right)\\
&\quad +\frac{(n-1)(n-5)}{12 n(m+2)}\left(\binom{n-1}{m+1}+(-1)^{m+1}\binom{n-1}{2 m+3}\right)\\
&\quad +\frac{m+2}{n(m+3)}\left(\binom{n-1}{m+2}+(-1)^{m}\binom{n-1}{2 m+5}\right)\\
\mathcal R_m^{(8)}&=e_1^{(6)}e_m^{(2)}-\mathcal R_{m+1}^{(6)}\\
&=-\frac{(n-1)(2 n^5+2 n^4-355 n^3-355 n^2+11153 n-19087)}{12\cdot 7!n(m+1)}\\
&\qquad\times\left(\binom{n-1}{m}+(-1)^m\binom{n-1}{2 m+1}\right)\\
&\quad -\frac{(n-1)(n^3+n^2-109 n+251)}{6!n(m+2)}\left(\binom{n-1}{m+1}+(-1)^{m+1}\binom{n-1}{2 m+3}\right)\\
&\quad -\frac{(n-1)(n-5)}{12 n(m+3)}\left(\binom{n-1}{m+2}+(-1)^{m}\binom{n-1}{2 m+5}\right)\\
&\quad -\frac{m+3}{n(m+4)}\left(\binom{n-1}{m+3}+(-1)^{m+1}\binom{n-1}{2 m+7}\right)\,.
\end{align*} 
The same recurrence gives explicit polynomial expressions for $\sum_{j=1}^m R_m^{(A,j)}(n)$ for every even $A$.  

For example, when $m=2$, we have 
\begin{align*}
&\mathfrak Z_n(\zeta_n;;2,4)+\mathfrak Z_n(\zeta_n;;4,2)=-\frac{(n-1)(n-2)(5 n^4-27 n^3-469 n^2+5787 n-13936)}{12\cdot 7!}\\
&\mathfrak Z_n(\zeta_n;;2,6)+\mathfrak Z_n(\zeta_n;;6,2)\\
&=\frac{(n-1)(n-2)(7 n^6-39 n^5-946 n^4+7950 n^3+33743 n^2-411111 n+773596)}{10!}\,. 
\end{align*}

\subsection{Odd case}

For odd values of $A$, we instead use an explicit expression of $\mathcal R_m^{(1)}$, which appears in \cite[Remark (14) of Theorem 6]{DK26}.   

\begin{Lem}  
\begin{equation}
\mathcal R_m^{(1)}
=\frac{1}{n}\left(\binom{n-1}{m}+(-1)^{m-1}\binom{n-1}{2 m}\right)\,.
\label{eq:222221}
\end{equation}
\label{lem:22122}  
\end{Lem}

By applying (\ref{eq:22a22}), together with Lemma \ref{lem:22122} (\ref{eq:222221}), we have  
\begin{align*}  
\mathcal R_m^{(3)}&=e_1^{(1)}e_m^{(2)}-\mathcal R_{m+1}^{(1)}\\
&=\frac{n-1}{2}\frac{1}{n(m+1)}\left(\binom{n-1}{m}+(-1)^m\binom{n-1}{2 m+1}\right)\\
&\quad -\frac{1}{n}\left(\binom{n-1}{m+1}+(-1)^{m}\binom{n-1}{2 m+2}\right)\\
&=-\frac{n-2 m-1}{2 n^2}\binom{n}{m+1}+\frac{(2 m+1)(-1)^m}{n^2}\binom{n}{2 m+2}
\,.
\end{align*}  

Applying (\ref{eq:22a22}) as $A=4$ and $A=6$, we have 
\begin{align*}  
\mathcal R_m^{(5)}&=e_1^{(3)}e_m^{(2)}-\mathcal R_{m+1}^{(3)}\\
&=-\frac{(n-1)(n-3)}{8 n(m+1)}\left(\binom{n-1}{m}+(-1)^{m}\binom{n-1}{2 m+1}\right)\\
&\quad +\frac{n-2 m-3}{2 n^2}\binom{n}{m+2}-\frac{(2 m+3)(-1)^{m+1}}{n^2}\binom{n}{2 m+4}\,,\\
\mathcal R_m^{(7)}&=e_1^{(5)}e_m^{(2)}-\mathcal R_{m+1}^{(5)}\\
&=\frac{(n-1)(n-5)(n^2+6 n-19)}{288 n(m+1)}\left(\binom{n-1}{m}+(-1)^m\binom{n-1}{2 m+1}\right)\\
&\quad +\frac{(n-1)(n-3)}{8 n(m+2)}\left(\binom{n-1}{m+1}+(-1)^{m+1}\binom{n-1}{2 m+3}\right)\\
&\quad -\frac{n-2 m-5}{2 n^2}\binom{n}{m+3}+\frac{(2 m+5)(-1)^{m}}{n^2}\binom{n}{2 m+6}\,. 
\end{align*}

\subsection{Expressions in terms of degenerate Bernoulli numbers}   

As in Theorem \ref{th:2}, a general expression for $\mathcal R_m^{(A)}$ can be obtained in terms of the degenerate Bernoulli numbers defined in (\ref{def:dbn}). The binomial coefficients outside their natural range are recognized as zero.  

\begin{theorem}  
For even $A$ with $A\ge 2$ and positive integers $n$ and $m$, we have 
\begin{align*}
&\sum_{j=1}^{m}\mathfrak Z_n(\zeta_n;;\underbrace{2,\dots,2}_{j-1},A,\underbrace{2,\dots,2}_{m-j})\\
&=\sum_{k=0}^{\frac{A}{2}-2}\sum_{j=1}^{A-2 k-2}\frac{(-1)^{k+1}n^j\beta_j(n^{-1})}{n(m+k+1)j!}\binom{A-2 k-3}{j-1}\left(\binom{n-1}{m+k}+(-1)^{m+k}\binom{n-1}{2 m+2 k+1}\right)\\
&\quad +(-1)^{\frac{A}{2}-1}\frac{m+\frac{A}{2}-1}{n(m+\frac{A}{2})}\left(\binom{n-1}{m+\frac{A}{2}-1}+(-1)^{m+\frac{A}{2}-1}\binom{n-1}{2 m+A-1}\right)\,. 
\end{align*} 
For odd $A$ with $A\ge 1$, we have 
\begin{align*}
&\sum_{j=1}^{m}\mathfrak Z_n(\zeta_n;;\underbrace{2,\dots,2}_{j-1},A,\underbrace{2,\dots,2}_{m-j})\\
&=\sum_{k=0}^{\frac{A-3}{2}}\sum_{j=1}^{A-2 k-2}\frac{(-1)^{k+1}n^j\beta_j(n^{-1})}{n(m+k+1)j!}\binom{A-2 k-3}{j-1}\left(\binom{n-1}{m+k}+(-1)^{m+k}\binom{n-1}{2 m+2 k+1}\right)\\
&\quad +(-1)^{\frac{A-1}{2}}\frac{1}{n}\left(\binom{n-1}{m+\frac{A-1}{2}}+(-1)^{m+\frac{A-3}{2}}\binom{n-1}{2 m+A-1}\right)\,. 
\end{align*} 
\label{th:2222}
\end{theorem}
\begin{proof}
When $A=2A'$ is even with $A'\ge 2$,  we have 
\begin{align*}
&\mathcal R_m^{(2 A')}\\
&=e_1^{(2 A'-2)}e_m^{(2)}-\mathcal R_{m+1}^{(2 A'-2)}\\
&=\sum_{k=0}^{A'-2}(-1)^k e_1^{(2 A'-2 k-2)}e_{m+k}^{(2)}+(-1)^{A'-1}\mathcal R_{m+A'-1}^{(2)}\\
&=\sum_{k=0}^{A'-2}(-1)^k(-1)\sum_{j=1}^{2 A'-2 k-2}\binom{2 A'-2 k-3}{j-1}\beta_j(n^{-1})\frac{n^j}{j!}\\
&\qquad\times\frac{1}{n(m+k+1)}\left(\binom{n-1}{m+k}+(-1)^{m+k}\binom{n-1}{2 m+2 k+1}\right)\\
&\quad +(-1)^{A'-1}\frac{m+A'-1}{n(m+A')}\left(\binom{n-1}{m+A'-1}+(-1)^{m+A'-1}\binom{n-1}{2 m+2 A'-1}\right)\\
&=\sum_{k=0}^{A'-2}\sum_{j=1}^{2 A'-2 k-2}\frac{(-1)^{k+1}n^j\beta_j(n^{-1})}{n(m+k+1)j!}\binom{2 A'-2 k-3}{j-1}\left(\binom{n-1}{m+k}+(-1)^{m+k}\binom{n-1}{2 m+2 k+1}\right)\\
&\quad +(-1)^{A'-1}\frac{m+A'-1}{n(m+A')}\left(\binom{n-1}{m+A'-1}+(-1)^{m+A'-1}\binom{n-1}{2 m+2 A'-1}\right)\,. 
\end{align*} 
The case where $A=2$ can be checked separately using $\mathcal R_m^{(2)}=m e_m^{(2)}$.  

When $A=2A'+1$ is odd with $A'\ge 1$,  we have 
\begin{align*}
&\mathcal R_m^{(2 A'+1)}\\
&=e_1^{(2 A'-1)}e_m^{(2)}-\mathcal R_{m+1}^{(2 A'-1)}\\
&=\sum_{k=0}^{A'-1}(-1)^{k+1}e_1^{(2 A'-2 k-1)}e_{m+k}^{(2)}+(-1)^{A'}\mathcal R_{m+A'}^{(1)}\\
&=\sum_{k=0}^{A'-1}(-1)^{k+1}\sum_{j=1}^{2 A'-2 k-1}\binom{2 A'-2 k-2}{j-1}\beta_j(n^{-1})\frac{n^j}{j!}\\
&\qquad\times\frac{1}{n(m+k+1)}\left(\binom{n-1}{m+k}+(-1)^{m+k}\binom{n-1}{2 m+2 k+1}\right)\\
&\quad +(-1)^{A'}\frac{1}{n}\left(\binom{n-1}{m+A'}+(-1)^{m+A'-1}\binom{n-1}{2 m+2 A'}\right)\\
&=\sum_{k=0}^{A'-1}\sum_{j=1}^{2 A'-2 k-1}\frac{(-1)^{k+1}n^j\beta_j(n^{-1})}{n(m+k+1)j!}\binom{2 A'-2 k-2}{j-1}\left(\binom{n-1}{m+k}+(-1)^{m+k}\binom{n-1}{2 m+2 k+1}\right)\\
&\quad +(-1)^{A'}\frac{1}{n}\left(\binom{n-1}{m+A'}+(-1)^{m+A'-1}\binom{n-1}{2 m+2 A'}\right)\,. 
\end{align*} 
This is also valid for $A'=0$ because the first term is nullified and the second term is the right-hand side of the identity (\ref{eq:222221}) in Lemma \ref{lem:22122}.  
\end{proof}

\section{Comments and future works}   
     
The recurrence strategy developed in this paper applies to the sum over the exceptional position. 
It seems possible to derive explicit formulas for the sum over $j$ of  
\begin{equation}
\sum_{1\le i_1<\dots<i_m\le n-1}u_{i_1}^3\cdots u_{i_{j-1}}^3 u_{i_j}^A u_{i_{j+1}}^3\cdots u_{i_m}^3\,.
\label{eq:33a33}
\end{equation}
Once explicit formulas for the basic cases $A=1,2,3$ in (\ref{eq:33a33}) are established, the same strategy may lead to a systematic treatment of this broader class of sums.  

A fixed-position evaluation $\mathfrak Z_n(\zeta_n;;\underbrace{1,\dots,1}_{j-1},A,\underbrace{1,\dots,1}_{m-j})$ for each $j$ is a separate question.


\section*{Acknowledgement}  


\end{document}